\documentclass{amsart}
\usepackage{mathtools}
\usepackage{amsmath}
\usepackage{amsthm}
\usepackage{mathrsfs}
\usepackage{amsmath}
\usepackage{amssymb}
\usepackage{amsthm}

\DeclareMathOperator{\Ind}{Ind}

\DeclareMathOperator{\Ric}{Ric}

\DeclareMathOperator{\Sym}{Sym}

\def\sideremark#1{\ifvmode\leavevmode\fi\vadjust{\vbox to0pt{\vss
 \hbox to 0pt{\hskip\hsize\hskip1em
 \vbox{\hsize3cm\tiny\raggedright\pretolerance10000
 \noindent #1\hfill}\hss}\vbox to8pt{\vfil}\vss}}}

\newcommand{\comment}[1]{}

\newtheorem{thm}{Theorem}[section]

\newtheorem{lem}[thm]{Lemma}
\newtheorem{cor}[thm]{Corollary}

\theoremstyle{definition}
\newtheorem{defn}[thm]{Definition}

\theoremstyle{remark}

\begin{document}

\title[Nonuniqueness for $\sigma_4$ and $H_4$]{Nonuniqueness for a fully nonlinear, degenerate elliptic boundary value problem in conformal geometry}
\author{Zhengyang Shan}

\email{szy199749@gmail.com}
\begin{abstract}
 We study the problem of conformally deforming a manifold with boundary to have vanishing $\sigma_4$-curvature in the interior and constant $H_4$-curvature on the boundary. We prove that there are geometrically distinct solutions using bifurcation results proven by Case, Moreira and Wang. Surprisingly, our construction via products of a sphere and hyperbolic space only works for a finite set of dimensions.
\end{abstract}
\maketitle

\renewcommand{\theequation}{\arabic{section}.\arabic{equation}}

\section{Introduction}
\label{sec:intro}

In this paper we use bifurcation theory to give a nonuniqueness result for a fully nonlinear, degenerate elliptic boundary-value problem involving the $\sigma_k$-curvature. Our result gives the first explicit examples of nonuniqueness for $k=4$, and relies on the general bifurcation theorem proven by Case, Moreira and Wang~\cite{Case}. We refer to the introduction of the article~\cite{Case} for a thorough account of the history of this problem in the context of nonuniqueness results for Yamabe-type problems.

Recall that the $\sigma_4$-curvature of a Riemannian manifold is defined by $\sigma_{4}=\sigma_{4}(g^{-1}P)$, where $P$ is the Schouten tensor. S. Chen~\cite{12} introduced the invariant $H_{4}$ on the boundary so that the pair $(\sigma_{4}; H_{4})$ is variational. See ~\cite{Case} for more details. 

We are interested in the set of elliptic solutions of the boundary-value problem 
\label{equ:main}
\begin{equation}
	\begin{cases}
		\sigma_{4}^{g}=0 , & \text{ in } X, \\ H_{4}^g = 1, & \text{ on } M
	\end{cases}
\end{equation}
in a given conformal class $[g_0]$ on $X$ and $g$ locally conformally flat. A solution is elliptic if it lies in the $C^{1,1}$-closure of
\[\Gamma_{4}^{+} = \lbrace g \in [g_0] \ | \ \sigma_{1}^{g} > 0, \dotsb , \sigma_{4}^g > 0 \rbrace. \]
Written in terms of a fixed background metric, Equation (1.1) is a fully nonlinear degenerate elliptic PDE with fully nonlinear Robin-type boundary condition.

The Case, Moreira and Wang~\cite{Case} bifurcation theorem involves the following Dirichlet problem. Suppose $T_{3} \coloneqq \frac{\partial \sigma_4}{\partial A_{i,j}}$ is positive definite. Then standard elliptic theory~\cite{20} implies that there exists a unique solution to
\label{equ2.5}
\begin{equation}
	\begin{cases}
		\delta(T_{3}(\nabla \upsilon))=0, & \text{ in } X, \\ \upsilon = \phi, & \text{ for all } \phi \in C^{\infty}(M). 
	\end{cases}
\end{equation}
This is the bifurcation theorem which gives sufficient conditions to conclude that a family of solutions to (1.1) has a bifurcation instant. 

\begin{thm}[{\cite{Case}}]
 \label{thm:main_thm}
Fix $4 \le j \in \mathbb{N}$ and $\alpha \in (0,1)$. Let $X^{n+1}$ be a compact manifold with boundary $M^{n} \coloneqq \partial X.$ Let $\lbrace g_{s} \rbrace_{s \in [a,b]}$ be a smooth one-parameter family of $C^\infty$-metrics on $X$ such that $\sigma_{4}^{g_s}=0$ and with respect to which $M$ has unit volume and constant $H_4$-curvature for all $s \in [a,b]$. We assume additionally that $g_s$ is locally conformally flat for all $s \in [a,b].$ Suppose that:\\
\indent (1) for every $s \in [a,b],$ the metric $g_{s} \in \overline{\Gamma^{+}_{4}}$ and there is a metric $\hat{g}_s \in \Gamma^{+}_{4}$  conformal to $g_{s}$ and such that $g_{s|TM} = \hat{g}_{s|TM};$ \\
\indent (2) for every $s \in [a,b], T_{3}^{g_s} > 0 \text{ and } S_{3}^{g_s}>0;$ \\
\indent (3) the Jacobi operators $\mathcal{DF}^{g_a} \text{ and } \mathcal{DF}^{g_b}$ are nondegenerate; and\\
\indent (4) $\Ind (\mathcal{DF}^{g_a}) \ne \Ind(\mathcal{DF}^{g_b}).$ \\
\noindent Then there exists an instant $s_* \in (a,b)$ and a sequence $(s_\ell)_\ell \subset [a,b]$ such that $s_\ell \rightarrow s_* \text{ as } \ell \rightarrow \infty$ and for each $\ell$, there are nonisometric unit volume $C^{j,\alpha}$-metrics in $[g_{s_\ell}|_{TM}]$ with constant $\mathcal{H}_4$-curvature.
\end{thm}

The function $\mathcal{F}$ is defined by
\[ \mathcal{F}(u) = \left (\sigma_{4}^{g_u}, H_{4}^{g_u} - \frac{1}{\text{vol}_{g_{u}}(M)} \oint_M H_{4}^{g_u} \text{dvol}_h \right ), \] 
where $g_u = e^{2u}g$. The Jacobi operator $\mathcal{DF}: C^{\infty}(M)\rightarrow C^{\infty}(M)$ is defined by restricting the linearization of $\mathcal{F}$ to solutions of (1.2). In particular, if $\mathcal{F}(1)=0$, then $\mathcal{DF}: C^{\infty}(M)\rightarrow C^{\infty}(M)$ is given by
\[ \mathcal{DF}(\phi) = T_{3}(\eta, \nabla \phi)- \overline{\delta}(S_{3}(\overline{\nabla} \phi))-7H_{4} \phi. \]

The Jacobi operator $\mathcal{DF}$ is \emph{nondegenerate} if 0 is not an eigenvalue of $\mathcal{DF}: C^{\infty}(M)\rightarrow C^{\infty}(M)$. The $\emph{index}$ Ind($\mathcal{DF}^g$) of $\emph{the Jacobi operator}$ is the number of negative eigenvalues of $\mathcal{DF}^g: C^{\infty}(M) \rightarrow C^{\infty}(M)$. 

The instant $s_*$ in Theorem~\ref{thm:main_thm} is in fact a bifurcation instant. 

\begin{defn}
\label{bifurcation}
Let $X^{n+1}$ be a compact manifold with nonempty boundary $M^{n} \coloneqq \partial X.$ Fix $j \ge 4$ and a parameter $\alpha \in (0,1).$ Let ${\lbrace g_s \rbrace}_{s \in [a,b]}$ be a smooth one-parameter family of $C^{j,\alpha}$-matrics on $X$ such that $\sigma_{4}^{g_s} = 0$ and with respect to which $M$ has unit volume and constant $H_4$-curvature. A $\emph{bifurcation instant}$ $\emph{for the family}$  ${\lbrace g_s \rbrace}$ is an instant $s_{*} \in (a,b)$ such that there exist sequences $(s_\ell)_\ell \subset [a,b]$ and $(w_\ell)_\ell \subset C^{j,\alpha}$ such that \\
\indent (1) $\sigma_{4}^{g_\ell} = 0$ and $H_{4}^{g_\ell}$ is constant, where $g_{\ell} \coloneqq e^{2w_{\ell}}g_{s_\ell}$, \\
\indent (2) $w_{\ell} \ne 0$ for all $\ell \in \mathbb{N}$, \\
\indent (3) $s_{\ell} \rightarrow s_{*}$ as $\ell \rightarrow \infty$, \\
\indent (4) $w_{\ell} \rightarrow 0$ in $C^{j,\alpha}$ as $\ell \rightarrow \infty$. \\
In particular, if $s_*$ is a bifurcation instant for a family ${\lbrace g_s \rbrace}$ of metrics as in Definition~\ref{bifurcation}, then for each $\ell \in \mathbb{N}$, there are nonhomothetic metrics in each conformal class $[g_{s_\ell}]$ which lie in $\overline{\Gamma^{+}_{4}}$, have $\sigma_4 = 0$, and have $H_4$ constant.
\end{defn}

Our first result is a nonuniqueness theorem on products of a spherical cap and a hyperbolic manifold. 

\begin{thm}
\label{thm:thm1.2}
Let $(S_{\varepsilon}^{806},d\theta^2), \ \varepsilon \in (0, \pi/2)$, be a spherical cap, let $(H^{715}, g_H)$ be a compact hyperbolic manifold, and denote by $(X_\varepsilon, g)$ their Riemannian product.\
Then, up to scaling, $({X_{\varepsilon}},g)$ is a solution of (1.1) for all $\varepsilon \in (0,\pi/2).$ Moreover, up to scaling, there is a sequence $(\varepsilon_j)_j \subset (0,\pi/2)$ of bifurcation instants for (1.1) for which $\varepsilon_{j} \rightarrow 0 \text{ as } j \rightarrow \infty.$
\end{thm}

Our second result is a nonuniqueness theorem on products of a round sphere and a small geodesic ball in hyperbolic space. 

\begin{thm}
\label{thm:thm1.3}
Let $(S^{806},d\theta^2)$ be a round sphere and let $(H_{\varepsilon}^{715}, g_H)$, $\varepsilon \in \mathbb{R_+}$, be a geodesic ball in hyperbolic space. Denote by $(X_{\varepsilon},g)$ their Riemannian product. \
Then, up to scaling, $({X_{\varepsilon}},g) \text{ is a solution of (1.1) for all } \varepsilon \in \mathbb{R_+}.$ Moreover, up to scaling, there is a sequence $(\varepsilon_{j})_{j} \subset \mathbb{R_+}$ of bifurcation instants for (1.1) for which $\varepsilon_{j} \rightarrow 0 \text{ as } j \rightarrow \infty. $
\end{thm}

One might wonder there is an infinite family of pairs $(m,n)$ such that the Riemannian product $S^{n} \times H^{m}$ satisfies $\sigma_k =0$. This is shown in ~\cite{Case} to be the case if $k \le 3$. By running the Algcurves(genus) package in Maple, we see that the genus of $\lbrace(m,n): \sigma_4(S^{n} \times H^{m}) = 0 \rbrace$ is three. Thus no such family can exist. 

We use Mathematica to calculate $n=806, m=715$. See details in Section 3.

\section{Background}
\label{sec:Backg}

In this section, we recall the definition of the $\sigma_k$-curvature and its essential properties.

Given $k \in \mathbb{N}$, the $\emph{k-th elementary symmetric function}$ of a symmetric $d\times d$-matrix $B \in \Sym_{d}$ is 
\[ \sigma_k(B) \coloneqq \sum_{i_{1} < \cdots < i_{k}} \lambda_{i_{1}} \cdots \lambda_{i_{k}}, \]
where $\lambda_{i_{1}}, \cdots , \lambda_{i_{k}}$ are the eigenvalues of $B.$ We compute $\sigma_{k}(B)$ via the formula 
\begin{equation}
\label{sigmak}
	\sigma_{k}(B) = \frac{1}{k!} \delta_{i_{1} \cdots i_{k}}^{j_{1} \cdots j_{k}} B_{j_{1}}^{i_{1}} \cdots B_{j_{k}}^{i_{k}},
\end{equation}
where $\delta_{i_{1} \cdots i_{k}}^{j_{1} \cdots j_{k}}$ denotes the generalized Kronecker delta,
\[
	\delta_{i_{1} \cdots i_{k}}^{j_{1} \cdots j_{k}} \coloneqq \begin{cases}
		1, & \text{if } (i_{1} \cdots i_{k}) \text{ is an even permutation of } (j_{1} \cdots j_{k}), \\ 
		-1, & \text{if } (i_{1} \cdots i_{k}) \text{ is an odd permutation of } (j_{1} \cdots j_{k}), \\
		0, & \text{otherwise},
	\end{cases}
\]
and Einstein summation convention is employed. The $\emph{k-th Newton tensor} \text{ of } B$ is the matrix $T_{k}(B) \in \Sym_d $ with components
\begin{equation}
\label{Tk}
	T_{k}(B)_{i}^{j} \coloneqq \frac{1}{k!} \delta_{ii_{1}...i_{k}}^{jj_{1}...j_{k}} B_{j_{1}}^{i_{1}}...B_{j_{k}}^{i_{k}}.
\end{equation}

Given nonnegative integers $k,\ell$ with $k \ge \ell$ and matrices $B, C \in \Sym_{d}, \text{we define}$

$$  \sigma_{k,\ell}(B,C) \coloneqq \frac{1}{k!} \delta_{i_{1} \cdots i_{k}}^{j_{1} \cdots j_{k}} {B_{j_{1}}^{i_{1}}} \cdots B_{j_{\ell}}^{i_{\ell}} C_{j_{\ell+1}}^{i_{\ell+1}} \cdots C_{j_{k}}^{i_{k}}, $$ 
$$ T_{k,\ell}(B,C)_{i}^{j} \coloneqq \frac{1}{k!} \delta_{ii_{1} \cdots i_{k}}^{jj_{1} \cdots j_{k}} {B_{j_{1}}^{i_{1}}} \cdots B_{j_{\ell}}^{i_{\ell}} C_{j_{\ell+1}}^{i_{\ell+1}} \cdots C_{j_{k}}^{i_{k}}. $$
That is, $\sigma_{k,\ell}(B,C) \ (\text{resp. } T_{k,\ell}(B,C))$ is the polarization of $\sigma_k \ (\text{resp. } T_k)$ evaluated at $\ell$ factors of $B \text{ and } k-\ell$ factors of $C$. 

The $\emph{positive k-cone}$ is
\[  \Gamma_{k}^{+} \coloneqq \lbrace B \in \Sym_n \ | \ \sigma_{1}(B), \cdots, \sigma_{k}(B) > 0 \rbrace \]
and its closure is
\[  \overline{\Gamma_{k}^{+}} \coloneqq \lbrace B \in \Sym_n \ | \ \sigma_{1}(B), \cdots, \sigma_{k}(B) \ge 0 \rbrace. \]
Their significance is that $T_{k-1}(B)$ is positive definite  (resp. nonnegative definite) for all $B \in \Gamma_{k}^{+} \ (\text{resp. all } B \in \overline{\Gamma_{k}^{+}})$ and that $\Gamma_{k}^{+} \text{ and } \overline{\Gamma_{k}^{+}}$ are convex~\cite{10}.

The $\emph{Schouten tensor}$ $P$ of $(X^{n+1},g)$ is the section
\[ P \coloneqq \frac{1}{n-1} \left( \text{Ric} - \frac{R}{2n} g\right) \]
of $S^{2}T^{*}X$, where Ric and R are the Ricci tensor and scalar curvature, respectively, of $g$. The $\sigma_{k}$-$\emph{curvature}$ of $(X, g)$ is 
\[ \sigma_{k}^{g} \coloneqq \sigma_{k}(g^{-1}P), \]
where $g^{-1}$ is the musical isomorphism mapping $T^{*}X$ to $TX$ and its extension to tensor bundles.

\begin{defn}
A Riemannian metric $g$ is $k$-$admissible$ if $g \in \overline{\Gamma^{+}_{k}}$ and there is a metric $\hat{g} \in \Gamma^{+}_{k}$  conformal to $g$ and such that $g\rvert_{TM} = \hat{g}\rvert_{TM}.$
\end{defn}

Suppose now that ($X^{n+1},g$) is a compact Riemannian manifold with boundary $M^{n} = \partial X$ which has unit volume with respect to the induced metric $h \coloneqq \iota^{*}g$. Denote by $h^{-1}$ the musical isomorphism mapping $T^{*}M \text{ to } TM$ and its extension to tensor bundles. The $H_k$\emph{-curvature} of $M$ is  
\[ H_{k}^{g} \coloneqq \sum_{j=0}^{k-1} \frac{(2k-j-1)!(n+1-2k+j)!}{j!(n+1-k)!(2k-2j-1)!!} \sigma_{2k-j-1,j} \ (h^{-1} \iota^{*}P, h^{-1}A). \]
and
\[ S_{k-1} \coloneqq \sum_{j=0}^{k-2} \frac{(2k-j-3)!(n+2-2k+j)!}{j!(n+1-k)!(2k-2j-3)!!} T_{2k-j-3,j} \ (h^{-1} \iota^{*}P, h^{-1}A). \]
The following corollary is a consequence of Theorem~\ref{thm:main_thm}; see ~\cite{Case}.
\begin{cor}
\label{cor5.7}
Fix $4 \le j \in \mathbb{N}$ and $\alpha \in (0,1).$ Let $a \in \mathbb{R}_+$ and denote by $\overline{B}^{n+1}(a)$ the closed ball of radius $a$ in $\mathbb{R}^{n+1}.$ Let $(N^m,g_N)$ be a compact Einstein manifold and suppose that there is an odd smooth function $f: (-a,a) \rightarrow \mathbb{R}$ and an even smooth function $\psi: (-a,a) \rightarrow \mathbb{R}_+$ such that
\[ g \coloneqq dr^{2} \oplus f^{2}(r) \ d\theta^{2} \oplus \psi^{2}(r) \ g_N \]
defines a locally conformally flat metric on $X \coloneqq \overline{B}^{n+1}(a) \times N^m$ such that $g \in \overline{\Gamma^{+}_4}$ and $\sigma^{g}_4 =0,$ where $r(x) = |x|$ for $x \in \overline{B}^{n+1}(a).$ Given $s \in (0,a),$ set
\[ X_s \coloneqq \lbrace (x,y) \in X \ | \ r(x) \le s \rbrace \]
and let $g_s$ denote the restriction of $g$ to $X_s$. Assume that there are $s_{1},s_{2} \in (0,a)$ such that $s_{1} < s_{2}$ and: 
\begin{enumerate}
\item for every $s \in [a,b],$ the metric $g_{s}|_{TM}$ is $4$-admissible; 
\item for every $s \in [a,b],$ it holds that $T_{3}^{g_s} > 0$ and $S_{3}^{g_s} > 0;$
\item $\ker \mathcal{DF}^{g_{s_1}}, \ker \mathcal{DF}^{g_{s_2}} \subset \mathbb{R},$ where $\mathbb{R}$ denotes the space of constant functions; and
\item $ \Ind(\mathcal{DF}^{g_{s_1}}) \ne \Ind(\mathcal{DF}^{g_{s_2}})$ when computed on $\mathbb{R}^{\perp}.$ 
\end{enumerate}
Then $\partial X_{s}$ has constant $H_4$-curvature for all $s \in (0,a)$, and there exists a bifurcation instant $s_{*} \in (s_{1},s_{2})$ for the family $(X_{s},g_{s})$.
\end{cor}

\section{Computations}
\label{sec:Comp}

In this section, we describe the method that we use to compute results for Section 4. 

We want to find the solutions that satisfy $\sigma_4=0$. We use Mathematica to solve the following equation:
\[ \sigma_4(A_{m,n}) = \binom{m}{4} - \binom{m}{3}\binom{n}{1} + \binom{m}{2}\binom{n}{2} - \binom{m}{1}\binom{n}{3} + \binom{n}{4} =0 \]

\noindent where $A_{m,n}$ is a diagonal matrix with entries $-1$ and $1$. The number of negative eigenvalues is $m$ and the number of positive eigenvalues is $n$. 

Here is a full list of solutions in terms of $(m,n)$ which satisfy $\sigma_4=0$ when $m <10000$: 
\begin{equation}
\label{soln}
(1,1), \ (1,2), \ (1,7), \ (3,5), \ (7,10), \ (30,36), \ (715,806), \ (7476, 7567)
\end{equation}
We also consider all solutions $(m,n)$ which satisfied $\sigma_5=0$. Excluding the trivial solutions, namely $m=n$, this is the full list of solutions when $m<1000$: 
\[ (1,2), \ (1,3), \ (1,9), \ (3,7), \ (3,14), \ (14,22), \ (22,45), \ (28,39) \ (133,156) \]
In the context of Theorem~\ref{thm:thm1.2} and ~\ref{thm:thm1.3}, we restrict our attention to pairs in $\eqref{soln}$ with $n+m$ strictly greater than 8 because of nonuniqueness results for the $\sigma_k$-curvature~\cite{26,24,25}. We also look for pairs which satisfy the ellipticity condition $\sigma_k \ge 0$, $k < 4$. The only such pair is (715, 806). We use this pair to compute everything in the rest of our paper. 

Given a diagonal matrix $B$, we let $B(i_\ell)$ be the entry on the $(i_\ell, i_\ell)$ component. If $B$ and $C$ are simultaneously diagonalized, then
\begin{equation}
\label{sigma}
\sigma_{k,\ell}(B,C) = \frac{1}{k!} \sum_{i_{1}, \cdots ,i_{k} \\ \text{distinct}} B(i_1) \cdots B(i_\ell) C(i_{\ell+1}) \cdots C(i_k)
\end{equation}
\begin{equation}
\label{T}
T_{k,\ell}(B,C)_{i}^{i} = \frac{1}{k!} \sum_{i,i_{1}, \cdots ,i_{k} \text{distinct}} B(i_1) \cdots B(i_\ell) C(i_{\ell+1}) \cdots C(i_k)
\end{equation}
Here are the computations that are relevant for the rest of the paper. We are interested in two cases, one is where we remove a negative eigenvalue, the other one is where we remove a positive eigenvalue.

A \emph{$(m,n)$-block diagonal matrix} is an $(m+n) \times (m+n)$ matrix with an $m \times m$ block $\lambda I_m$ in the upper left, an $n \times n$ block $\mu I_n$ in the lower right, and zeros everywhere else. We denote such a matrix by $\lambda I_m \oplus \mu I_n$.
\begin{lem}
\label{comp_lem1}
Suppose $B$ is the $(m,n-1)$ matrix obtained from $A_{m,n}$ by removing the last column and last row, and suppose $C$ is the $(m,n-1)$ block diagonal matrix with $(\lambda, \mu) = (0,\kappa)$. Then
\end{lem}
\begin{flalign*}
&\sigma_{j,0} (B,C) = \binom{n-1}{j}\kappa^{j},\\
&T_{j,0} = \binom{n-1}{j}\kappa^{j}I_{m} \oplus \binom{n-2}{j}\kappa^{j}I_{n}
\end{flalign*}
for any $j \in \mathbb{N}$. Moreover, 
\begin{flalign*}
&\sigma_{6,1}(B,C)= \frac{(n-m-6)(n-5)(n-4)(n-3)(n-2)(n-1)}{6!}\kappa^{5} \\
&\sigma_{5,2}(B,C)= \frac{(n-3)(n-2)(n-1)(n^{2}+m^{2}-2mn-9n+7m+20)}{5!}\kappa^{3} \\
&\sigma_{4,3}(B,C)= \frac{(n-m-2)(n-1)(n^2+m^2-2mn-7n+m+12)}{4!}\kappa \\
&T_{4,1} = \frac{(n-m-3)(n-3)(n-2)(n-1)}{4!}\kappa^3I_{n} \\
& \qquad \oplus \frac{(n-m-5)(n-4)(n-3)(n-2)}{4!}\kappa^3 I_{m} \\
&T_{3,2} = \frac{(n-1)(n^2+m^2-2mn-3n+m+4)}{3}\kappa I_{n} \\
& \qquad \oplus \frac{(n-2)(n^2+m^2-2mn-7n+5m+12)}{3}\kappa I_{m}
\end{flalign*}
\begin{proof}
We perform the above calculations by counting $k,\ell$ from equations $\eqref{sigma}$ and $\eqref{T}$ where $\ell$ is the number of eigenvalues we get from $B$ and $k-\ell$ is the number of eigenvalues we get from $C$. 
\end{proof}
\begin{lem}
\label{comp_lem2}
Suppose $B$ is the $(m-1,n)$ matrix obtained from $A_{m,n}$ by removing the first column and first row, and suppose $C$ is the $(m-1,n)$ block diagonal matrix with $(\lambda, \mu) = (\kappa,0)$. Then
\end{lem}
\begin{flalign*}
&\sigma_{j,0}(B,C) = \binom{m-1}{j}\kappa^{j},\\
&T_{j,0} = \binom{m-1}{j}\kappa^{j}I_{m} \oplus \binom{m-2}{j}\kappa^{j}I_{n}
\end{flalign*}
for any $j \in \mathbb{N}$. Moreover,
\begin{flalign*}
&\sigma_{6,1}(B,C)= \frac{(m-n-6)(m-5)(m-4)(m-3)(m-2)(m-1)}{6!}\kappa^{5} \\
&\sigma_{5,2}(B,C)= \frac{(m-3)(m-2)(m-1)(m^{2}+n^{2}-2mn-9m+7n+20)}{5!}\kappa^{3} \\
&\sigma_{4,3}(B,C)= \frac{(m-n-2)(m-1)(m^2+n^2-2mn-7m+n+12)}{4!}\kappa \\
&T_{4,1} = \frac{-(m-n-5)(m-4)(m-3)(m-2)}{4!}\kappa^3I_{n} \\
& \qquad \oplus \frac{-(m-n-3)(m-3)(m-2)(m-1)}{4!}\kappa^3I_{m} \\
&T_{3,2} = \frac{(m-2)(m^2+n^2-2mn-7m+5n+12)}{3}\kappa I_{n} \\
& \qquad \oplus \frac{(m-1)(m^2+n^2-2mn-3m+n+4)}{3}\kappa I_{m}
\end{flalign*}
\begin{proof}
We perform the above calculations by counting $k,\ell$ from equations $\eqref{sigma}$ and $\eqref{T}$ where $\ell$ is the number of eigenvalues we get from $B$ and $k-\ell$ is the number of eigenvalues we get from $C$. 
\end{proof}

\section{Proofs of Theorem~\ref{thm:thm1.2} and Theorem~\ref{thm:thm1.3}}
\label{sec:proofs}

We begin by considering the interior geometry of certain Riemannian products.

\begin{lem}
 \label{lem:basic_lemma}
Let $(M^{806},g_M) \text{ and } (H^{715},g_H)$ be Einstein manifolds with $\Ric_{g_M} =805g_M \text{ and } \Ric_{g_H} = -714g_H,$ respectively, and let $(X^{1521},g)$ denote their Riemannian product.
Then $(X,g)$ is such that 
\[ \sigma_1 = \frac{91}{2}, \ \sigma_2=\frac{3380}{4}, \ \sigma_3=\frac{56420}{8}, \ \sigma_4 \equiv 0. \]
Moreover, 
\[ T_3 = \frac{483}{52}(715g_M \oplus 806g_H).\]
\end{lem}

\begin{proof}
Let $e_1, \dotsb ,e_{806}$ be a basis for $T_{p}M$ and $f_1, \dotsb , f_{715}$ be a basis for $T_{q}H$. Then at $(p,q) \in M \times H,$ with respect to the basis $e_1, \dotsb ,e_{806}, f_1, \dotsb , f_{715}$ of $T_{(p,q)} M \times H$,
\[ g^{-1} P = \frac{1}{2} A_{715,806} \]
The computations of $\sigma_4$ and $T_3$ readily follow.
\end{proof}

\subsection{Products of a spherical cap and a hyperbolic manifold}
In this subsection we apply Corollary~\ref{cor5.7} to products of a spherical cap and a hyperbolic manifold with sectional curvature $1$ and $-1$, respectively. Note that this normalization ensures that the product is locally conformally flat. Our first task is to study the geometry of the boundary of these products. 
\begin{lem}
 \label{lem:lemma1}
Denote $(S^{806},d\theta^2)$ and $(H^{715}, g_H)$ 
the round $806$-sphere of constant sectional curvature 1 and $715$-dimensional hyperbolic manifold of constant sectional curvature $-1$, respectively. Given $\varepsilon \in (0,\pi/2)$, set
\[ S_\varepsilon^{806} = \lbrace x \in S^{806} \ | \ r(x) \le \varepsilon \rbrace, \]
where $r$ is the geodesic distance from a fixed point $p \in S^{806}$. Let $(X_\varepsilon^{1521},g)$
denote the Riemannian product of $(S_\varepsilon^{806},d\theta^2)$ and $(H^{715}, g_H)$, and let $\iota$ denote the inclusion of 
$H$ into $\partial X_\varepsilon$. Let $\kappa = \cot \varepsilon$ denote the mean curvature of $\partial S_\varepsilon^{806}$ in $S_\varepsilon^{806}$.

\noindent Then $(X_\varepsilon,g)$ is such that 
\begin{equation}
\sigma_1 = \frac{91}{2}, \ \sigma_2=\frac{3380}{4}, \ \sigma_3=\frac{56420}{8}, \ \sigma_4 \equiv 0, \ T_3 > 0.
\end{equation}
Moreover $g|_{T\partial X_\varepsilon}$ is 4-admissible and the boundary $\partial X_\varepsilon$ is such that $H_4$ is a nonnegative constant, $S_3>0$, and
\[ H_4 = \frac{11,194,421,414,880}{28,977,203}\kappa^{7} + \mathcal{O}(\kappa^{5}) \] 
\[ \iota^{*}S_2 = \frac{927,410,178,387}{144,886,015}\kappa^{5} + \mathcal{O}(\kappa^{3}) \]
as $\varepsilon \rightarrow 0$
\end{lem}

\begin{proof}
The claims about the $\sigma_{4}$-curvatures and the Newton tensors follow from Lemma~\ref{lem:basic_lemma}. We prove the rest of Lemma~\ref{lem:lemma1} following the same strategy of Case, Moreira and Wang~\cite{Case}.

We write the metric $g$ on $X_\varepsilon$ as
\[ g = dr^2\oplus \sin^{2}rd\vartheta^2 \oplus g_H. \]

Fix $s \in \mathbb{R_+}$ and define $u: S_\varepsilon^{806} \times H \rightarrow \mathbb{R}$ by
\[ u(p,q)=\frac{1+sr^2(p)}{1+s\varepsilon^2}. \]

Set $g_u := u^{-2}g,$ 
\[ P^{g_u}=\frac{1+4s}{2}dr^2 \oplus \frac{1+4sr\cot r}{2}\sin^{2}rd\vartheta^2 \oplus \left( -\frac{1}{2} \right)g_H + \mathcal{O}(s^2)\]
for $s$ close to zero. Therefore
\[ g_{u}^{-8}\sigma_4^{g_u} = \sigma_4^g+\frac{1}{4}s\sum_{j=0}^{3} (-1)^{j} \binom{805}{3-j} \binom{715}{j} (1+805r\cot r)+\mathcal{O}(s^2). \]
It follows that $g_u \in \Gamma^+_4$ for $s$ sufficiently close to zero. Thus $g|_{T\partial X_\varepsilon}$ is 4-admissible.

By definition, 
\[ H_4 = \frac{2}{219,212,540,695}\sigma_{7,0}+\frac{2}{144,886,015}\sigma_{6,1}+ \frac{1}{114,837}\sigma_{5,2} + \frac{1}{379}\sigma_{4,3},\]
and
\[ S_3 =\frac{1}{434,658,045}T_{5,0} + \frac{2}{574,185}T_{4,1}+\frac{3}{1,516}T_{3,2}. \]
Combining these formulae with Lemma~\ref{comp_lem1} yields the claimed conclusions for $H_4$ and $S_3.$ 
\end{proof}

Here is the proof for Theorem~\ref{thm:thm1.2}.
\begin{proof}[Proof of Theorem~\ref{thm:thm1.2}]
Applying Lemma~\ref{lem:lemma1} to $({X_{\varepsilon}},g)$ implies that, up to scaling, $({X_{\varepsilon}},g)$ is a solution of (1.1) for all $\varepsilon \in (0, \pi /2)$. Lemma~\ref{lem:lemma1} further implies that there are constants $c_1, c_2>0$ such that $\iota_{2}^{*}S_{3} = {c_1}\varepsilon^{-5}g_H +\mathcal{O}(\varepsilon^{-3})$ and $H_{4}=c_{2}\varepsilon^{-7}+\mathcal{O}(\varepsilon^{-5})$ as $\varepsilon \rightarrow 0, \text{where } \iota: H \rightarrow \partial X_\varepsilon$ is inclusion map. Let $\pi: \partial X_{\varepsilon} \rightarrow H$ denote the projection map. As noted in \cite{Case}, for all $\phi \in C^{\infty}(H),$ the extension $v_{\phi} \text{ of } \pi^{*}\phi \text{ to } X_{\varepsilon}$ by (1.2) is of the form $v_{\phi}(p,q)=f(r(q))\phi(p).$ Therefore $T_{3}(\eta,\nabla v_{\phi})= \mathcal{O}(1) \text{ as } \varepsilon \rightarrow 0.$ Thus
\[ \mathcal{DF}^{g}(\pi^{*}\phi)=\pi^{*}[-\delta_{g_H}((\iota^{*}S_{3})(\overline{\nabla}\phi))-7(\iota^{*}H_{4})\phi]+\mathcal{O}(1). \]
for all $\phi \in C^{\infty}(H).$ It follows that the index of $\mathcal{DF}$ tends to $\infty$ as $\varepsilon \rightarrow 0.$ Corollary~\ref{cor5.7} then yields, up to scaling, the existence of the sequence $(\varepsilon_j)_j$ of bifurcation instants.
\end{proof}

\subsection{Products of a round sphere and a small geodesic ball in hyperbolic space}
In this subsection we apply Corollary~\ref{cor5.7} to products of a round sphere and a small geodesic ball in hyperbolic space with sectional curvature $1$ and $-1$, respectively. Note that this normalization ensures that the product is locally conformally flat. Our first task is to study the geometry of the boundary of these products. 
\begin{lem}
 \label{lem:lemma}
Denote $(S^{806},d\theta^2)$ and $(H^{715}, g_H)$ 
the round $ 806 $-sphere of constant sectional curvature 1 and the $ 715 $-dimensional  simply connected manifold of constant sectional curvature $-1$, respectively. Given $\varepsilon \in (0,\pi/2)$, set
\[ H_\varepsilon^{715} = \lbrace x \in H^{715} \ | \ r(x) \le \varepsilon \rbrace, \]
where $r$ is the geodesic distance from a fixed point $p \in H^{715}$. Let $(X_\varepsilon^{1521},g)$
denote the Riemannian product of $(S^{806},d\theta^2)$ and $(H^{715}, g_H)$, and let $\iota$ denote the inclusion of 
$S^{806}$ into $\partial X_\varepsilon$. Let $\kappa = \coth \varepsilon$ denote the mean curvature of $\partial H_\varepsilon^{715}$ in $H_\varepsilon^{715}$. 
Then $(X_\varepsilon,g)$ is such that
\begin{equation}
\sigma_1 = \frac{91}{2}, \ \sigma_2=\frac{3380}{4}, \ \sigma_3=\frac{56420}{8}, \ \sigma_4 \equiv 0, \ T_3 > 0
\end{equation}
Moreover $g|_{T\partial X_\varepsilon}$ is 4-admissible and the boundary $\partial X_\varepsilon$ is such that $H_4$ is a nonnegative constant, $S_3>0$, and
\[ H_4 = \frac{24,089,939,471,088}{144,886,015}\kappa^{7} + \mathcal{O}(\kappa^{5}) \] 
\[ \iota^{*}S_2 = \frac{508,268,486,964}{144,886,015}\kappa^{5} + \mathcal{O}(\kappa^{3}) \]
as $\varepsilon \rightarrow 0$
\end{lem}

\begin{proof}
The claims about the $\sigma_{4}$-curvatures and the Newton tensors follow from Lemma~\ref{lem:basic_lemma}. We prove the rest of Lemma~\ref{lem:lemma} following the same strategy of Case, Moreira and Wang~\cite{Case}.

We write the metric $g$ on $X_\varepsilon$ as
\[ g = d\theta^2\oplus dr^2\oplus \sinh^2rd\vartheta^2. \]

Fix $s \in \mathbb{R_+}$ and define $u: S^{806} \times H_\varepsilon^{715} \rightarrow \mathbb{R}$ by
\[ u(p,q)=\frac{1+sr^2(q)}{1+s\varepsilon^2}. \]

Set $g_u := u^{-2}g.$ 
\[ P^{g_u}=\frac{1}{2}d\theta^2\oplus\frac{4s-1}{2} dr^2\oplus\frac{4sr\text{coth}r-1}{2}\sinh^{2}rd\vartheta^2+\mathcal{O}(s^2) \]
for $s$ close to zero. Therefore
\[ \sigma_4^{g_u} = \sigma_4^g+\frac{1}{4}s\sum_{j=0}^{3} (-1)^{3-j} \binom{714}{3-j} \binom{806}{j} (1+714r\text{coth}r)+\mathcal{O}(s^2). \]
It follows that $g_u \in \Gamma^+_4$ for $s$ sufficiently close to zero. Thus $g|_{T\partial X_\varepsilon}$ is 4-admissible, as appropriate.

By definition, 
\[ H_4 = \frac{2}{219,212,540,695}\sigma_{7,0}+\frac{2}{144,886,015}\sigma_{6,1}+ \frac{1}{114,837}\sigma_{5,2} + \frac{1}{379}\sigma_{4,3},\]
and
\[ S_3 =\frac{1}{434,658,045}T_{5,0} + \frac{2}{574,185}T_{4,1}+\frac{3}{1,516}T_{3,2}. \]
Combining these formulae with Lemma~\ref{comp_lem2} yields the claimed conclusions for $H_4$ and $S_3.$ 
\end{proof}

Here is the proof for Theorem~\ref{thm:thm1.3}.
\begin{proof}[Proof of Theorem~\ref{thm:thm1.3}]
Applying Lemma~\ref{lem:lemma} to $({X_{\varepsilon}},g)$ implies that, up to scaling, $({X_{\varepsilon}},g)$ is a solution of (1.1) for all $\varepsilon \in (0, \pi /2)$. Lemma~\ref{lem:lemma} further implies that there are constants $c_1, c_2>0$ such that $ \iota^{*}S_{3} = {c_1}\varepsilon^{-5}g_H +\mathcal{O}(\varepsilon^{-3})$ and $H_{4}=c_{2}\varepsilon^{-7}+\mathcal{O}(\varepsilon^{-5})$ as $\varepsilon \rightarrow 0, \text{where } \iota: S^{806} \rightarrow \partial X_\varepsilon$ is inclusion map. Let $\pi: X_{\varepsilon} \rightarrow S^{806}$ denote the projection map. As noted in \cite{Case}, for all $\phi \in C^{\infty}(S^{806}),$ the extension $v_{\phi} \text{ of } \pi^{*}\phi \text{ to } X_{\varepsilon}$ by (1.2) is of the form $v_{\phi}(p,q)=f(r(q))\phi(p).$Therefore $T_{3}(\eta,\nabla v_{\phi})= \mathcal{O}(1) \text{ as } \varepsilon \rightarrow 0.$ Thus
\[ DF^{g}(\pi^{*}\phi)=\pi^{*}[-\delta_{d\theta^{2}}((\iota^{*}S_{3})(\overline{\nabla}\phi))-7(\iota^{*}H_{4})\phi]+\mathcal{O}(1). \]
for all $\phi \in C^{\infty}(S^{806}).$ It follows that the index of $\mathcal{DF}$ tends to $\infty$ as $\varepsilon \rightarrow 0.$ Corollary~\ref{cor5.7} then yields, up to scaling, the existence of the sequence $(\varepsilon_j)_j$ of bifurcation instants.
\end{proof}

\section*{Acknowledgement} The author would like to thank her advisor Dr. Jeffrey Case for many helpful discussions.

\bibliographystyle{abbrv}
\bibliography{bib}
\end{document}